\documentclass{amsart}
\usepackage{amsmath,amssymb,amsthm,url,xcolor}
\newtheorem{thr}{Theorem}

\newtheorem{obs}[thr]{Observation}
\newtheorem{claim}[thr]{Claim}

\theoremstyle{definition}

\theoremstyle{remark}

\numberwithin{equation}{section}

\def\E{\mathcal{E}}

\def\E{\mathcal{E}}

\begin{document}%\color{blue}

%\title[An improved bound for the length of matrix algebras]{An improved bound for the \\ length of matrix algebras}
%\title{Counterexamples to the product conjecture of S. Hedetniemi}
\title{Counterexamples to Hedetniemi's conjecture}

%    author one information
\author{Yaroslav Shitov}
%\address{National Research University Higher School of Economics, 20 Myasnitskaya Ulitsa, Moscow 101000, Russia}
\email{yaroslav-shitov@yandex.ru}

%    \subjclass is required.
%\subjclass[2010]{15A23}

%\keywords{Nonnegative matrix factorization, NP-hard problem}

%\date{October 3, 2011 and, in revised form, February 22, 2012}

%\dedicatory{}

%    "Communicated by" -- provide editor's name; required.
%\commby{Jim Haglund}

%    Abstract is required.

\begin{abstract}
The chromatic number of $G\times H$ can be less than the minimum of the chromatic numbers of finite simple graphs $G$ and $H$.
\end{abstract}

\maketitle

The \textit{tensor product} $G\times H$ of finite simple graphs $G$ and $H$ has vertex set $V(G)\times V(H)$, and pairs $(g,h)$ and $(g',h')$ are adjacent if and only if $\{g,g'\}\in E(G)$ and $\{h,h'\}\in E(H)$. One can easily see that $\chi(G\times H)\leqslant\chi(G)$ because a proper coloring $\Psi$ of the graph $G$ can be lifted to the coloring $(g, h)\to\Psi(g)$ of $G\times H$. Similarly, a proper coloring of $H$ leads to a proper coloring of $G\times H$ with the same number of colors, so we get
%\begin{equation*}\label{Heq}\chi(G\times H)\leqslant \min\{\chi(G), \chi(H)\}.\eqno{H}\end{equation}
$$\chi(G\times H)\leqslant \min\{\chi(G), \chi(H)\}.\eqno{\mathrm{(H)}}$$

The classical conjecture of S. T. Hedetniemi~\cite{Hedetniemi} posited the equality for all $G$ and $H$. More than 50 years have passed since the conjecture appeared, and it keeps attracting serious attention of researchers working in graph theory and combinatorics; we mention four exhaustive survey papers~\cite{Klavzar, Sau, Tard, Zhu} for more detailed information on the topic. Here, we briefly recall that Hedetniemi's conjecture was proved in many special cases, including graphs with chromatic number at most four~\cite{EZS}, graphs containing large cliques~\cite{BEL, DSW, Wel}, circular graphs and products of cycles~\cite{VPV}, and Kneser graphs and hypergraphs~\cite{HM}. The conjecture gave an impetus to the study of \textit{multiplicative graphs}, which remains remarkably active and important in its own right~\cite{HHMNL, Tard2, Wrochna}. A generalization of Hedetniemi's conjecture to fractional chromatic numbers turned out to be true~\cite{Zhu2}, but the version with directed graphs is false~\cite{PR}, as well as the one with infinite chromatic numbers~\cite{Haj, Rin}. We show that the inequality~(H) can be strict for finite simple graphs.

\bigskip

%\section{Coloring large exponential graphs}

A standard tool in the study of Hedetniemi's conjecture is the concept of the \textit{exponential graph} as introduced in~\cite{EZS}. Let $c$ be a positive integer, and let $\Gamma$ be a finite graph that we allow to contain loops; the graph $\E_c(\Gamma)$ has all mappings $V(\Gamma)\to\{1,\ldots,c\}$ as vertices, and two distinct mappings $\varphi,\psi$ are adjacent if, and only if, the condition $\varphi(x)\neq\psi(y)$ holds whenever $\{x,y\}\in E(\Gamma)$. The relevance of $\E_c(\Gamma)$ to the problem is easy to see because the graph $\Gamma\times\E_c(\Gamma)$ has the proper $c$-coloring $(h,\psi)\to\psi(h)$. The idea of our approach lies in the fact that the proper $c$-colorings of $\E_c(\Gamma)$ become quite well-behaved if the graph $\Gamma$ is fixed and $c$ gets large; let us proceed to technical details and exact statements.
A basic result in~\cite{EZS} tells that the constant mappings form a $c$-clique in $\E_c(\Gamma)$, which means that these mappings get different colors in a proper $c$-coloring. So a relabeling of colors can turn any proper $c$-coloring $\Psi:\E_c(\Gamma)\to\{1,\ldots,c\}$ into a \textit{suited} one, in which a color $i$ is assigned to the constant mapping sending every vertex of $\Gamma$ to $i$.

\begin{obs}\label{obs1}
If $\Psi$ is a suited proper $c$-coloring of $\E_c(\Gamma)$, then $\Psi(\varphi)\in\operatorname{Im}\varphi$.
\end{obs}

\begin{proof}
The mapping $\varphi$ is adjacent to the constant mapping $\{v\to j\}$ for any $j$ not in $\operatorname{Im}\varphi$, so $\varphi$ cannot get colored with such a $j$.
\end{proof}

%In the rest of this section, we consider a fixed graph $H=(V,E)$ that may contain loops, and we study the behavior of $c$-colorings of the exponential graph $\E_c(H)$ as $c$ gets large. The $O$-notation of the following statements is understood relative to the base $c\to+\infty$, so in particular we write $n:=|V|\in O(1)$. We call an independent set $I$ of $\E_c(H)$ \textit{large} if $|I|\in\Omega(c^{n-1})$ and \textit{small} otherwise.

\begin{claim}\label{lem3}
Consider a graph $\Gamma$ with $|V(G)|=n$ and a suited proper $c$-coloring $\Psi$ of $\E_c(\Gamma)$. Then there is a vertex $v=v(\Psi)$ of $\Gamma$ such that $\geqslant c-\sqrt[n]{n^3c^{n-1}}$ color classes $\Psi^{-1}(b)$ are $v$\textit{-robust}, which means that, for any $\varphi_b\in\Psi^{-1}(b)$, there is a $w\in N(v)$ satisfying $\varphi_b(w)=b$, where $N(v)$ stands for the closed neighborhood of $v$ in $\Gamma$.
\end{claim}

\begin{proof}
For any color $b$ and any vertex $u\in V(\Gamma)$, we define $I(u,b)$ as the set of all $\varphi\in \Psi^{-1}(b)$ that satisfy $\varphi(u)=b$. According to Observation~\ref{obs1}, every vertex of $\E_c(\Gamma)$ belongs to at least one of the classes $I(u,b)$.

Assume that $I(u,b)$ is a \textit{large} class, that is, it contains more than $n^2c^{n-2}$ elements, and consider an arbitrary mapping $\varphi_b\in\Psi^{-1}(b)$. If every element $\psi_{ub}$ of $I(u,b)$ admitted a vertex $u'\neq u$ with $\psi_{ub}(u')\in\operatorname{Im}\varphi_b$, then there would be at most $n^2$ ways to choose $u'$ and $\psi_{ub}(u')$, while the remaining $n-2$ vertices would contribute at most a factor of $c^{n-2}$. This contradicts the cardinality assumption on $I(u,b)$, so we can actually find a $\psi_{ub}\in I(u,b)$ under which $u$ is an only vertex taking the color $b$ and also $\operatorname{Im}\varphi_b\cap\operatorname{Im}\psi_{ub}=\{b\}$. In other words, the equality $\Psi(\varphi)=b$ cannot hold unless there is a vertex $w\in N(u)$ satisfying $\varphi(w)=b$.

%The set $I(u,b)$ is to be called \textit{large} if it contains more than $n^2c^{n-2}$ elements

%, and the color $b$ is \textit{fair} if $I(u,b)$ is large for at least one vertex $u$. According to Observation~\ref{obs1}, the set $\Psi^{-1}(b)$ equals the union of all $I(u,b)$ over all $u\in V(\Gamma)$, so an unfair color covers at most $n^3c^{n-2}$ vertices of $\E_c(\Gamma)$.

If there is a vertex $v\in V(\Gamma)$ for which $I(v,b)$ is large for at least $c-\sqrt[n]{n^3c^{n-1}}$ colors $b$, then we are done. Conversely, we can define more than $n^3c^{n-1}$ mappings $\varphi: V(\Gamma)\to\{1,\ldots,c\}$ for which the value of $\varphi$ on a vertex $w$ does not equal those colors $b$ for which $I(w,b)$ is large. None of these mappings belongs to a large class $I(u,b)$, but the non-large classes are too small to cover all of them.
%If there are at least $c-\sqrt[n]{n^3c^{n-1}}$ fair colors $b$ for which the sets $N(u_b)$ intersect at some vertex $v$, then the proof is complete. Otherwise, we can define more than $n^3c^{n-1}$ mappings $\varphi: V(\Gamma)\to\{1,\ldots,c\}$ for which the value of $\varphi$ on a vertex $w$ can equal only those colors $b$ that satisfy $w\notin N(u_b)$; these mappings cannot get a fair color under $\Psi$ according to the previous sentence. However, the unfair classes cannot cover this number of mappings according to the last sentence of the first paragraph.
\end{proof}

\bigskip

Now we are ready to proceed with counterexamples. For a simple graph $G$, we define the graph $\Gamma_G$ by adding the loops to all the vertices, and the \textit{strong product} $G\boxtimes K_q$ as the graph with vertex set $V(G)\times\{1,\ldots,q\}$ and edges between $(u,i)$ and $(v,j)$ when, and only when, $\{u,v\}\in E(G)$ or $(u=v)\&(i\neq j)$.

\begin{claim}\label{claim4}
Let $G$ be a finite simple graph with finite girth $\geqslant 6$. Then, for sufficiently large $q$, one has $\chi\left(\E_c(G\boxtimes K_q)\right)>c$ with $c=\lceil 3.1q\rceil$.
\end{claim}

\begin{proof}
The restriction of a suited proper coloring $\Lambda:\E_c(G\boxtimes K_q)\to\{1,\ldots,c\}$ to the mappings that are constant on the cliques $\{g\}\times K_q\subset G\boxtimes K_q$ is a proper coloring $\Psi:\E_c(\Gamma_G)\to\{1,\ldots,c\}$ up to the identification of every such clique with $g$. We find a vertex $v=v\left(\Psi\right)\in V(G)$  as in Claim~\ref{lem3} and define the clique $\mathcal{M}=\{\mu_{q+1},\ldots,\mu_c\}$ in $\E_c(G\boxtimes K_q)$ by setting, for all $i\in\{1,\ldots,q\}$ and $t\in\{q+1,\ldots,c\}$,

\smallskip

\noindent (1.1) $\mu_t(g, i)=i$ for all $g\in V(G)$ satisfying $\operatorname{dist}(v,g)\in\{0,2\}$,

\noindent (1.2) $\mu_t(g, i)=q+i$ for all $g\in V(G)$ satisfying $\operatorname{dist}(v,g)=1$,

\noindent (1.3) $\mu_t(g,i)=t$ for all $g\in V(G)$ satisfying $\operatorname{dist}(v,g)\geqslant 3$.

\smallskip

Due to the assumption on the girth of $G$, no pair of vertices defined in (1.1) and (1.2) can be adjacent in $G\boxtimes K_q$ and monochromatic at the same time; the condition (1.3) uses different colors for different $t$, and these colors are also different from those of the neighboring vertices dealt with in (1.1). Therefore, $\mathcal{M}$ is indeed a clique and requires $c-q\geqslant 2.1q$ colors. Using the pigeonhole principle, one finds a $\tau\in\{q+1,\ldots,c\}$ such that $\Lambda(\mu_\tau)\notin\{1,\ldots,2q\}$, and due to Observation~\ref{obs1} we have $\tau=\Lambda(\mu_\tau)$. Further, it is only $o(q)$ classes that are not $v$-robust with respect to $\Psi$ in the terminology of Claim~\ref{lem3}, so we can find a $v$-robust class $\sigma\notin\{1,\ldots,2q,\tau\}$. %we have $\Lambda(\mu_\tau)=\tau$ because of Observation~\ref{obs1}. Similarly, we can find a $v$-robust class $\sigma\in\{2q+1,\ldots,c\}$ different from $\tau$.
%distinct classes $\sigma, \tau\in\{2q+1,\ldots,c\}$ such that $\sigma$ is $v$-robust, and $\mu_\tau$ belongs to a $v$-robust class different from $1,\ldots,2q$; we have $\Lambda(\mu_\tau)=\tau$ because of Observation~\ref{obs1}.
Finally, we note that the mapping $\nu:G\boxtimes K_q\to\{1,\ldots,c\}$ defined as, for all $i$,

\smallskip

\noindent (2.1) $\nu(g,i)=\tau$ for all $g\in V(G)$ in the closed neighborhood $N(v)$,

\noindent (2.2) $\nu(g,i)=\sigma$ for all $g\in V(G)$ satisfying $\operatorname{dist}(v,g)\geqslant 2$,

\smallskip 

\noindent is adjacent to $\mu_\tau$ in $\E_c(G\boxtimes K_q)$. Since $\sigma$ is $v$-robust, we cannot have $\Lambda(\nu)=\sigma$ by Lemma~\ref{lem3}, but rather we have $\Lambda(\nu)=\tau$ according to Observation~\ref{obs1}. So we have $\Lambda(\nu)=\Lambda(\mu_\tau)$, which is a contradiction.
\end{proof}

The classical paper~\cite{Erd} proves the existence of graphs with arbitrarily large girth and fractional chromatic number; so we can find a graph $G$ of girth at least $6$ that satisfies $\chi_f(G)>3.1$. We set $c=\lceil 3.1q\rceil$ and pass to sufficiently large $q$; we immediately get $\chi(G\boxtimes K_q)\geqslant q\cdot\chi_f(G)>c$ and also $\chi\left(\E_c(G\boxtimes K_q)\right)>c$ by Claim~\ref{claim4}. The equality $\chi\left((G\boxtimes K_q)\times \E_c(G\boxtimes K_q)\right)=c$ follows by standard theory~\cite{EZS} as the mapping $(u,\varphi)\to\varphi(u)$ is a proper $c$-coloring of any graph of the form $\Gamma\times \E_c(\Gamma)$.

\end{document}